\documentclass[final,12pt]{elsarticle}
\usepackage{amsmath}
\usepackage{amsthm}
\usepackage{amssymb}
\usepackage{mathtools}
\usepackage{array}
\usepackage{booktabs}
\usepackage{bm}
\usepackage{tikz}
\usepackage{xcolor}
\usepackage{verbatim}
\usepackage{appendix}
\usepackage{enumerate}
\usepackage{enumitem}
\usepackage{xcolor}

\setlist[enumerate,1]{label=(\arabic*),font=\textup,
leftmargin=7mm,labelsep=1.5mm,topsep=0mm,itemsep=-0.8mm}
\setlist[enumerate,2]{label=(\alph*),font=\textup,
leftmargin=7mm,labelsep=1.5mm,topsep=-0.8mm,itemsep=-0.8mm}

\usepackage{geometry}
\usepackage{microtype}

\usepackage[pdfstartview=FitH,bookmarksopen=false,
colorlinks=true,linkcolor=blue,citecolor=blue]{hyperref}

\usepackage{graphicx}
\usepackage{floatrow}
\usepackage{natbib}
\newtheorem{theorem}{Theorem}[section]

\newtheorem{lemma}{Lemma}[section]

\numberwithin{equation}{section}
\newenvironment{proof1}{\noindent{\textbf{Proof.}}\ }{\hfill $\square$\par}
\newenvironment{proof of 1.1}{\noindent{\textbf{Proof of Theorem 1.1.}}\ }{\hfill $\square$\par}
\newenvironment{proof of 2.1}{\noindent{\textbf{Proof of Lemma 2.4.}}\ }{\hfill $\square$\par}
\newenvironment{proof of 2.2}{\noindent{\textbf{Proof of Lemma 2.5.}}\ }{\hfill $\square$\par}

\usepackage{indentfirst}
\begin{document}

\begin{frontmatter}
		\title{The maximum index of signed complete graphs whose negative edges induce a bicyclic graph\,}
		\author{Ziyi Fang$^{1,2}$}
        \author{Fan Chen$^{1,2}$}
		\author{Xiying Yuan$^{1,2}$\corref{corespond}}
		\cortext[corespond]{Corresponding author. \\
Email address: ziyi\_sunny@shu.edu.cn (Ziyi Fang), chenfan@shu.edu.cn (Fan Chen), \\xiyingyuan@shu.edu.cn (Xiying Yuan)\\ This work is supported by the National Natural Science Foundation of China (Nos. 11871040,
12271337, 12371347).}
		\address{$^1$Department of Mathematics, Shanghai University, Shanghai 200444, P.R. China}
		\address{$^2$Newtouch Center for Mathematics of Shanghai University, Shanghai 200444, P.R. China}
				
		\begin{abstract}
		 Let $\Gamma=(K_n,H)$ be a signed complete graph whose negative edges induce a subgraph $H$. Let $A(\Gamma)$ be the adjacency matrix of the signed graph $\Gamma$. The largest eigenvalue of $A(\Gamma)$ is called the index of $\Gamma$.  In this paper, the index of all the signed complete graphs whose negative edges induce a bicyclic graph $B$ is investigated. Specifically, the structure of the bicyclic graph $B$ such that $\Gamma=(K_n,B)$ has the maximum index is determined.
		\end{abstract}
		

\end{frontmatter}

\section{Introduction}
A signed graph $\Gamma=(G,\sigma)$ consists of an underlying graph $G=(V(G),E(G))$ and a signature $\sigma:E(G)\rightarrow\{1,-1\}$, where $G$ is a graph with the vertex set $V(G)$ and the edge set $E(G)$, and $\sigma$ is a mapping defined on the edge set of $G$.
Signed graphs were first introduced by Harary \cite{Harary1953On} and by Cartwright and Harary \cite{1956Structuralbalance}.
The adjacency matrix of the signed graph $\Gamma$ is denoted as  $A(\Gamma)=(a_{ij}^{\sigma})$, where $a_{ij}^{\sigma}=\sigma(v_iv_j)a_{ij}$, if $v_i\sim v_j$, and 0 otherwise. The characteristic polynomial of $A(\Gamma)$ is denoted by $\varphi(\Gamma,\lambda)$. By the eigenvalues, spectrum, and eigenvectors of $\Gamma$, we refer to those of $A(\Gamma)$. Let $\lambda_1(\Gamma)\geqslant\cdots\geqslant\lambda_n(\Gamma)$ be the eigenvalues of $\Gamma$. In particular, the largest eigenvalue $\lambda_1(\Gamma)$ is called the index. The spectrum of graphs and signed graphs has been studied in recent years, as seen in works such as \cite{Akbari2019On,2020OnAkbari,Belardo2016On}.

Let $E_v(G)=\{e|v\in e\in E(G)\}$. The degree of $v$ in $G$ is defined as $|E_v(G)|$. A vertex with degree 1 is referred to as a pendant vertex.
A connected signed graph $\Gamma$ is called $t-$cyclic if its underlying graph $G$ is $t-$cyclic with $|E(G)|=|V(G)|+t-1$. In particular, we designate a 1-cyclic signed graph $\Gamma$ as a signed unicyclic graph and a 2-cyclic signed graph $\Gamma$ as a signed bicyclic graph \cite{2018nongolden}. For more details about the notion of signed graphs, we refer to \cite{1982Signedgraphs}.

In \cite{Koledin2017Connected}, Koledin and Stani\'{c} studied the connected signed graphs of fixed order, size, and number of negative edges. They conjectured that if $\Gamma$ is a signed complete graph with $n$ vertices and $k$ negative edges, where $k<n-1$, then $\Gamma$ attains maximum index if and only if the negative edges induce a star $K_{1,k}$.
Akbari, Dalvandi, Heydari, and Maghasedi \cite{2019SignedAkbari} proved that the conjecture holds for signed complete graphs in which the negative edges form a tree.  Ghorbani and Majidi \cite{2021SignedEbrahim} completely confirmed the conjecture.

Let $d(u,v)$ be the length of the shortest path between vertex $u$ and vertex $v$. The diameter of a graph $G$ is defined as the maximum of $d(u,v)$ for any  $u,v\in V(G)$. Recently, Li, Lin, and Meng \cite{2023ExtremalLiDan} identified the signed graph with the maximum index and the smallest minimum eigenvalue among the signed complete graphs whose negative edges induce a tree of order $n$ and have a diameter of at least $3$.
In \cite{2024minimumeigenvalue}, Ghorbani and Majidi investigated the unique signed graph among signed complete graphs whose negative edges induce a tree $T$, where $T$ is defined as a tree with a diameter of at least $d$ for any given $d$. They also determined the smallest minimum eigenvalue of signed complete graphs with $n$ vertices and $k$ negative edges, where the negative edges induce a tree. In \cite{2021uncycleNavidKafai}, Kafai, Heydari, Red, and Maghasedi demonstrated that among all signed completed graphs of order $n>5$ whose negative edges induce a unicyclic graph of order $k$ and maximizes the index, the negative edges induce a triangle with all remaining vertices being pendant at the same vertex of the triangle. A positive cycle is defined as a cycle in a signed graph that contains an even number of negative edges, while a negative cycle is not positive. A signed graph is balanced if it contains no negative cycles; otherwise, it is unbalanced. Recently, Li, Lin, and Teng \cite{2024LiDanLinHUIQIu} investigated the maximum index of unbalanced signed complete graphs whose negative edges induce a tree of order $n$ with a given number of pendant vertices. In \cite{2025LaplacianLiDan}, Li, Yan, and Meng examined the least Laplacian eigenvalue of $(K_n,F)$, where $F$ is a unicyclic graph.

We denote the complete graph of order $n$ by $K_n$. In \cite{2019SignedAkbari}, $(K_n,H)$ denotes a signed complete graph of order $n$ whose negative edges induce the (unsigned) graph $H$.
In this paper, we consider the signed complete graphs whose negative edges induce a bicyclic graph $B$ and investigate the structure of the bicyclic graph $B$ to ensure that $(K_n,B)$ has the maximum index.

\begin{figure}[H]
\centering
\includegraphics[scale=0.7]{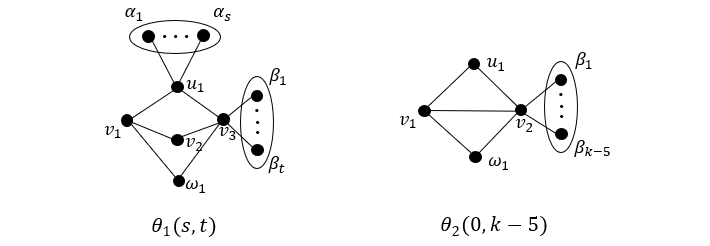}
\caption{Graphs $\theta_1(s,t)$ and $\theta_2(0,k-5)$.}
\label{fig:B11}
\end{figure}

We define the bicyclic graphs $\theta_1(s,t)$ and $\theta_2(0,k-5)$ with $k$ edges in Figure \ref{fig:B11}, where $s+t=k-6$.
Then we establish the following result.
\begin{theorem}\label{thm:maximumindexall}
Let $B$ be a bicyclic graph with $k$ edges. The signed complete graph $(K_n,B)$ attains the maximum index if and only if $B$ is isomorphic to $\theta_2(0,k-5)$ $($see Fig.1$)$ when $n\geqslant k+20$ and $k\geqslant 15$.
\end{theorem}

Let $\mathcal{G}=\{(K_n,\theta_1(s,t)),(K_n,\theta_2(0,k-5))\}$, where $s+t=k-6$. The graphs $\theta_1(s,t)$ and $\theta_2(0,k-5)$ are depicted in Figure \ref{fig:B11}. We will focus on the maximum index of signed complete graphs in $\mathcal{G}$ in Section \ref{Sec:mathcalG}. Additionally, we will give the proof of Theorem \ref{thm:maximumindexall} in Section \ref{Sec:maxindex}.

\section{The maximum index of signed graphs in $\mathcal{G}$}\label{Sec:mathcalG}

Let $M$ be a symmetric matrix of order $n$, and $\{X_1,\cdots,X_m\}$ be a partition of $\{1,\cdots,n\}$. We have the following block form
\begin{equation*}
\begin{split}
M=\left(
\begin{array}{ccc}
 M_{11} & \cdots & M_{1m}\\
 \vdots & \ddots & \vdots\\
 M_{m1} & \cdots & M_{mm}\\
\end{array}
\right).
\end{split}
\end{equation*}
The matrix $Q=(q_{ij})$ is called the quotient matrix of $M$, where $q_{ij}$ denotes the average row sum of $M_{ij}$, for $1\leqslant i,j\leqslant m$. Moreover, a partition is called equitable if for each pair $i,j$, the row sum of $M_{ij}$ remains constant. The matrix $Q$ is an equitable quotient matrix when the partition is equitable.

Let $\eta=\{X_1,\cdots,X_p,X_{p+1},\cdots,X_{p+q}\}$ be a special partition of the vertices of the signed completed graph $\Gamma$, where $|X_i|=n_i$ and $\Gamma[X_i]=(K_{n_i},+)$ for $i=1,\cdots,p$, $\Gamma[X_i]=(K_{n_i},-)$ for $i=p+1,\cdots,p+q$, and
all the edges have the same sign between $X_i$ and $X_j$ for each $i,j$. Obviously, the special partition $\eta$ is equitable.

\begin{lemma}\label{Lem:partitionquotient}\textup{(\cite{2021uncycleNavidKafai})} Let $\Gamma$ be a signed complete graph, $\eta=\{X_1,\cdots,X_p,X_{p+1},\cdots,X_{p+q}\}$ be a special partition of $V(\Gamma)$. Let $Q$ be the quotient matrix of $A(\Gamma)$ related to $\eta$. Then the characteristic polynomial of $A(\Gamma)$ is
\begin{equation*}
\begin{split}
\varphi(\Gamma,\lambda)=(\lambda+1)^{\sum_{i=1}^{p}n_i-p}(\lambda-1)^{\sum_{i=p+1}^{p+q}n_i-q}\varphi(Q,\lambda),
\end{split}
\end{equation*}
where $n_i=|X_i|$.
\end{lemma}

To simplify representation, let
\begin{equation*}
\begin{split}
F(\lambda)=&\lambda^7+(7-n)\lambda^6+(21-6n)\lambda^5+(4ku+8st-15n+16k+8s-61)\lambda^4\\
&+(16ku+32st-20n+64k+32s-349)\lambda^3+(-24ku-16st-16stu\\
&-16su+273n-192k+48s-267)\lambda^2+(-80ku-96st-32stu-32su\\
&+570n-512k+32s+199)\lambda-44ku-56st+112stu-16su+287n\\
&-272k+8s+193,
\end{split}
\end{equation*}
where $u=n-s-t-5=n-k+1$. Let $N_{G}(u)$ be the set of neighbours of $u$ in $G$. Then, we have following result.

\begin{lemma}\label{lem:BBB1}
    If $\Gamma=(K_n,\theta_{1}(s,t))$, where $s+t=k-6$. Then $\varphi(\Gamma,\lambda)=(\lambda+1)^{n-7}F(\lambda)$.
\end{lemma}

\begin{proof}
Assume that the set of vertices of the signed graph $\Gamma$ is partitioned into parts $X_1,\cdots,X_7$, where
$X_1=\{v_1\}$, $X_2=\{u_1\}$, $X_3=\{v_3\}$, $X_4=\{v_2,\omega_1\}$, $X_5=N_{\theta_1(s,t)}(v_3)\backslash\{u_1,v_2,\omega_1\}$, $X_6=N_{\theta_1(s,t)}(u_1)\backslash\{v_1,v_3\}$ and
$X_7=V(K_n)\backslash(X_1\cup X_2\cup X_3\cup X_4\cup X_5\cup X_6)$.

The equitable quotient matrix $Q$ of $A(\Gamma)$ related to $X_1,\cdots,X_7$ is
\begin{equation*}
\begin{split}
Q=\left(
\begin{array}{ccccccc}
0 & -1 & 1 & -2 & t & s & u\\
-1 & 0 & -1 & 2 & t & -s & u\\
1 & -1 & 0 & -2 & -t & s & u\\
-1 & 1 & -1 & 1 & t & s & u\\
1 & 1 & -1 & 2 & t-1 & s & u\\
1 & -1 & 1 & 2 & t & s-1 & u\\
1 & 1 & 1 & 2 & t & s & u-1
\end{array}
\right),
\end{split}
\end{equation*}
where $u=n-s-t-5=n-k+1$.

By direct calculation, the characteristic polynomial of $Q$ is
\begin{equation*}
\begin{split}
\varphi(Q,\lambda)=F(\lambda).
\end{split}
\end{equation*}

Since $\Gamma[X_i]=(K_{n_i},+)$ for $i=1,\cdots,7$ and
all the edges have the same sign between $X_i$ and $X_j$ for each $i,j$,
$\{X_1,\cdots,X_7\}$ is a special partition. By Lemma \ref{Lem:partitionquotient}, we have
$\varphi(\Gamma,\lambda)=(\lambda+1)^{n-7}\varphi(Q,\lambda).$
\end{proof}

To simplify representation, let
\begin{equation*}
\begin{split}
P(\lambda)=&\lambda^5+(5-n)\lambda^4+(10-4n)\lambda^3+(4ku-6n+8k-22)\lambda^2+(8ku-36n\\
&+48k-91)\lambda-28ku+127n-120k+97,
\end{split}
\end{equation*}
where $u=n-k+1$. Then, we have following result.

\begin{lemma}\label{lem:BBB2}
    If $\Gamma=(K_n,\theta_{2}(0,k-5))$. Then $\varphi(\Gamma,\lambda)=(\lambda+1)^{n-5}P(\lambda)$.
\end{lemma}

\begin{proof}
Assume that the set of vertices of the signed graph $\Gamma$ is partitioned into parts $X_1,\cdots,X_5$, such that
$X_1=\{v_1\}$, $X_2=\{v_2\}$, $X_3=\{u_1,\omega_1\}$, $X_4=N_{\theta_2(0,k-5)}(v_2)\backslash\{v_1,u_1,\omega_1\}$ and
$X_5=V(K_n)\backslash(X_1\cup X_2\cup X_3\cup X_4)$.

The equitable quotient matrix $Q$ of $A(\Gamma)$ related to $X_1,\cdots,X_5$ is
\begin{equation*}
\begin{split}
Q=\left(
\begin{array}{ccccccc}
0 & -1 & -2 & k-5 & u\\
-1 & 0 & -2 & 5-k & u\\
-1 & -1 & 1 & k-5 & u\\
1 & -1 & 2 & k-6 & u\\
1 & 1 & 2 & k-5 & u-1\\
\end{array}
\right),
\end{split}
\end{equation*}
where $u=n-k+1$.

By direct calculation, the characteristic polynomial of $Q$ is
\begin{equation*}
\begin{split}
\varphi(Q,\lambda)=P(\lambda).
\end{split}
\end{equation*}

Since $\Gamma[X_i]=(K_{n_i},+)$ for $i=1,\cdots,5$ and all the edges have the same sign between $X_i$ and $X_j$ for each $i,j$, the set
$\{X_1,\cdots,X_5\}$ is a special partition. By Lemma \ref{Lem:partitionquotient}, we have
$\varphi(\Gamma,\lambda)=(\lambda+1)^{n-5}\varphi(Q,\lambda).$
\end{proof}

\begin{lemma}\label{lem:BBBst1}
$\lambda_1((K_n,\theta_{1}(s,t)))\leqslant\max\{\lambda_1((K_n,\theta_{1}(0,k-6))),\lambda_1((K_n,\theta_{1}(k-6,0)))\}$, $k\geqslant15$, for each pair $s,t$, where $s+t=k-6$.
\end{lemma}

This Lemma is derived through direct calculation, and the proof can be found in \ref{App:A}.

\begin{lemma}\label{lem:BBBmax} Let $k\geqslant15$, $n\geqslant k+20$. \textup{(i)} $\lambda_1((K_n,\theta_{1}(0,k-6)))<\lambda_1((K_n,\theta_{2}(0,k-5)))$. \\
\textup{(ii)} $\lambda_1((K_n,\theta_{1}(k-6,0)))<\lambda_1((K_n,\theta_{2}(0,k-5)))$.
\end{lemma}

The calculation process of Lemma \ref{lem:BBBmax} is also included in \ref{App:A}.
Combining Lemma \ref{lem:BBBst1} and Lemma \ref{lem:BBBmax}, we can see that $\lambda_1((K_n,\theta_{2}(0,k-5)))$ is the maximum index of the signed graphs in $\mathcal{G}$.

\section{The maximum index of $(K_n,B)$}\label{Sec:maxindex}

In this section, we provide the proof of Theorem \ref{thm:maximumindexall}. Before proving the theorem, we need the following results.

\begin{theorem}\label{thm:interlacing}\textup{(\cite{2011spectraofgraphs})}
\textup(Cauchy Interlacing Theorem\textup) Let $A$ be a Hermitian matrix of order $n$ with eigenvalues $\lambda_1\geqslant\cdots\geqslant\lambda_n$ and $B$ be a principal submatrix of $A$ of order $m$ with eigenvalues $\mu_1\geqslant\cdots\geqslant\mu_m$, then $\lambda_i\geqslant \mu_i\geqslant\lambda_{n-m+i}$ for $i=1,\cdots,m$.
\end{theorem}
\begin{lemma}\label{lem:perturbation}\textup{(\cite{Koledin2017Connected})}
Let $r,s,t$ be distinct vertices of a signed graph $\Gamma$, $\mathbf{x}=(x_1,\cdots,x_n)^T$ be an eigenvector corresponding to the index $\lambda_1(\Gamma)$. Let $\Gamma'$ be obtained by reversing the sign of the positive edge $rt$ and the negative edge $rs$. If $x_r\geqslant0$, $x_s\geqslant x_t$, or $x_r\leqslant0$, $x_s\leqslant x_t$, then $\lambda_1(\Gamma')\geqslant\lambda_1(\Gamma)$, and if at least one inequality for the entries of $\mathbf{x}$ is strict, then $\lambda_1(\Gamma')>\lambda_1(\Gamma)$.
\end{lemma}

Suppose $\Gamma=(K_n,B)$ is a signed complete graph of order $n$ whose $k$ negative edges induce a bicyclic graph $B$ that maximizes the index. Let $\lambda_1=\lambda_1(\Gamma)$ and $\mathbf{x}$ be an eigenvector corresponding to $\lambda_1$. Since $(K_{n-k+2},+)$ is an induced subgraph of $\Gamma$, by Theorem \ref{thm:interlacing}, we have $\lambda_1\geqslant n-k+1$. We denote the cycle of length $l$ by $C_l$. Suppose $C_{l_1}$ and $C_{l_2}$ are the two vertex-induced cycles of $B$. Without loss of generality, we assume that $l_1\leqslant l_2$.

\begin{lemma}\label{Lem:Cl1Cl2} There exists a vertex $v_{i}$ in $C_{l_i}$ such that $x_{v_i}\neq0$, $i=1,2$.
\end{lemma}
\begin{proof1}
Without loss of generality, we consider the cycle $C_{l_1}$. Let $N_{B}^{r}(C_{l_1})$ be the set of vertices at distance $r$ from vertices of cycle $C_{l_1}$ in $B$. By the contrary, assume that $x_{v_i}=0$ for any $v_i$ in $V(C_{l_1})$.

We will show that $x_{v_r}=0$ for any $v_r\in V(B)$. Since $B$ is a bicyclic graph, we have $N_{B}^{1}(C_{l_1})$ is not empty. For any $v_r\in N_{B}^{1}(C_{l_1})$, let $v_p$ be a vertex in $V(C_{l_1})$ adjacent to $v_{r}$ in $B$ and $v_q$ be a vertex in $V(C_{l_1})$ not adjacent to $v_{r}$ in $B$. If $x_{v_r}\neq0$, since $x_{v_p}=x_{v_q}=0$, then we can construct a new signed graph $\Gamma'$ obtained from $\Gamma$ by reversing the sign of the positive edge $v_rv_q$ and the negative edge $v_rv_p$ whose negative edges also form a bicyclic graph $B'$, then by Lemma \ref{lem:perturbation}, $\lambda_1(\Gamma)<\lambda_1(\Gamma')$, which is a contradiction. Hence, we have $x_{v_r}=0$ for any $v_r\in N_{B}^{1}(C_{l_1})$. If $N_{B}^{2}(C_{l_1})=\emptyset$, then $V(B)=V(C_{l_1})\cup N_{B}^{1}(C_{l_1})$, and then $x_{v_r}=0$ for any $v_r\in V(B)$.

Next, we will only consider the case $N_{B}^{2}(C_{l_1})$ is not empty. For any $v_r\in N_{B}^{2}(C_{l_1})$, there is always a vertex $v_p$ in $N_{B}^{1}(C_{l_1})$ that is adjacent to $v_r$ in $B$ and there is also a vertex $v_q$ in $V(C_{l_1})$ that is not adjacent to $v_{r}$ in $B$. If $x_{v_r}\neq0$, since $v_p=v_q=0$, then we can construct a new signed graph $\Gamma'$ obtained from $\Gamma$ by reversing the sign of the positive edge $v_rv_q$ and the negative edge $v_rv_p$ whose negative edges also form a bicyclic graph $B'$, then by Lemma \ref{lem:perturbation}, $\lambda_1(\Gamma)<\lambda_1(\Gamma')$, which is a contradiction. Thus $x_{v_r}=0$ for any $v_r\in N_{B}^{2}(C_{l_1})$. Since $B$ is connected, by the same argument, we can prove that $x_{v_r}=0$, for any $v_r$ in $V(B)$.

Then we always have the components of $\mathbf{x}$ corresponding to the vertices in $V(B)$ are all zero. If $V(B)=V(\Gamma)$, then $\mathbf{x}=\mathbf{0}^T$, which is a  contradiction. If $V(B)\neq V(\Gamma)$, then $V(\Gamma)\backslash V(B)$ is not empty.
Let $x_{v_m}=\max\{x_{v_i}:x_{v_i}\in V(\Gamma)\backslash V(B)\}$. We assume that $x_{v_m}>0$ (otherwise, consider $-\mathbf{x}$ instead of $\mathbf{x}$). By the eigenvalue equation for $v_m$, we have
$$\lambda_1 x_{v_m}=\sum_{v_i\in V(\Gamma)\backslash\{v_m\}}x_{v_i}=\sum_{v_i\in V(\Gamma)\backslash \{V(B)\cup \{v_m\}\}}x_{v_i}\leqslant |V(\Gamma)\backslash \{V(B)\cup \{v_m\}\}| x_{v_m}\leqslant (n-k)x_{v_m},$$ which the second equality follows from $x_{v_r}=0$, for any $v_r$ in $V(B)$. Hence, we have $\lambda_1\leqslant n-k$, which is a contradiction to $\lambda_1\geqslant n-k+1$. Furthermore, $x_{v_r}\neq0$ for some vertex $v_{r}\in V(C_{l_1})$. Similarly, $x_{v_r}\neq0$ for some vertex $v_{r}\in V(C_{l_2})$.
\end{proof1}

\begin{lemma}$l_1\leqslant l_2\leqslant4$.
\end{lemma}
\begin{proof1}
Since $l_1\leqslant l_2$, we suffice to show $l_2\leqslant4$. Suppose to the contrary that $l_2\geqslant 5$. Let $C_{l_2}=v_1v_2\cdots v_{l_2-1}v_{l_2}v_1$. By Lemma \ref{Lem:Cl1Cl2}, there is a vertex $v_i\in V(C_{l_2})$ such that
 $x_{v_i}\neq 0$. Without loss of generality, suppose $x_{v_1}>0$ (otherwise, consider $-\mathbf{x}$ instead of $\mathbf{x}$).

If $x_{v_3}\leqslant x_{v_{l_2}}$ (resp. $x_{v_{l_2-1}}\leqslant x_{v_{2}}$), then we can construct a new signed graph $\Gamma'$ from $\Gamma$ by reversing the sign of the positive edge $v_1v_3$ (resp. $v_1v_{l_2-1}$) and the negative edge $v_1v_{l_2}$ (resp. $v_1v_2$) whose negative edges also form a bicyclic graph $B'$, then by Lemma \ref{lem:perturbation}, $\lambda_1(\Gamma)<\lambda_1(\Gamma')$, which is a contradiction. Thus we have $x_{v_{l_2}}<x_{v_3}$ and $x_{v_2}<x_{v_{l_2-1}}$. If $x_{v_{2}}\geqslant 0$ (resp. $x_{v_{l_2}}\geqslant 0$), then we construct a new signed graph $\Gamma'$ from $\Gamma$ by reversing the sign of the positive edge $v_2v_{l_2}$ (resp. $v_{l_2}v_2$) and the negative edge $v_2v_3$ (resp. $v_{l_2}v_{l_2-1}$) whose negative edges also form a bicyclic graph $B'$, then by Lemma \ref{lem:perturbation}, $\lambda_1(\Gamma)<\lambda_1(\Gamma')$, which is a contradiction. Thus we have
$x_{v_{2}}<0$ and $x_{v_{l_2}}<0$.

If $x_{v_{l_2-1}}\leqslant  x_{v_{3}}$, then we can construct a new signed graph $\Gamma'$ obtained from $\Gamma$ by reversing the sign of the positive edge $v_{l_2}v_3$ and the negative edge $v_{l_2}v_{l_2-1}$ whose negative edges also form a bicyclic graph $B'$, by Lemma \ref{lem:perturbation}, we have $\lambda_1(\Gamma)<\lambda_1(\Gamma')$, which is a contradiction. Thus we have $x_{v_3}<x_{v_{l_2-1}}$, and then we can construct a new signed graph $\Gamma'$ obtained from $\Gamma$ by reversing the sign of the positive edge $v_2v_{l_2-1}$ and the negative edge $v_2v_3$ whose negative edges also form a bicyclic graph $B'$, by Lemma \ref{lem:perturbation}, we have $\lambda_1(\Gamma)<\lambda_1(\Gamma')$, which is a contradiction.
Hence, we have $l_2\leqslant 4$, and then $l_1\leqslant l_2\leqslant4$.
\end{proof1}

Let $\hat{B}$ be the base of a bicyclic graph $B$, which is the (unique) minimal bicyclic subgraph of $B$. 
Let bicyclic graph $\hat{\theta}(l,k,m)$($1\leqslant l\leqslant k\leqslant m$) be obtained from three pairwise internal disjoint paths $v_1u_1\cdots u_{k-1}v_{l+1}$ with length $k$, $v_1v_2\cdots v_{l}v_{l+1}$ with length $l$ and $v_1\omega_1\cdots\omega_{m-1}v_{l+1}$ with length $m$ from a vertex $v_1$ to a vertex $v_{l+1}$ (see Figure \ref{fig:hatB1hatB2}).
\begin{figure}[H]
\centering
\includegraphics[scale=0.69]{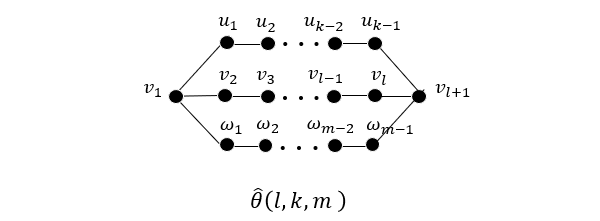}
\caption{Graph $\hat{\theta}(l,k,m)$.}
\label{fig:hatB1hatB2}
\end{figure}

\begin{proof of 1.1} Let $B$ be a bicyclic graph such that $(K_n,B)$ achieves the maximum index. We will prove that $B$ is isomorphic to $\theta_2(0,k-5)$ (see Figure \ref{fig:B11}).
Since $l_1\leqslant4$, we divide the proof into two cases.

\noindent\underline{\textbf{Case 1.}} $l_1=4$

\noindent\textbf{Claim 1.} $\hat{B}\cong\hat{\theta}(2,2,2)$

Suppose $C_{l_1}=v_1v_2v_3v_4v_1$. Assume that $v_1v_p\in E(\hat{B})$, where $v_p\in V(\hat{B})\backslash V(C_{l_1})$. If $v_p$ is adjacent to $v_3$ in $\hat{B}$, it is obvious that $\hat{B}\cong\hat{\theta}(2,2,2)$.
Suppose that $v_p$ is not adjacent to $v_3$ in $\hat{B}$. Since $l_2\geqslant l_1$, $v_p$ is not adjacent to $v_2$ and $v_4$. There is another vertex $v_q$ in $V(\hat{B})$ that is adjacent to $v_p$ in $\hat{B}$. Suppose that $v_q$ is not adjacent to $v_4$ (otherwise, we exchange the subscripts $v_2$ and $v_4$). Without loss of generality, suppose $x_{v_4}\geqslant 0$ (otherwise, consider $-\mathbf{x}$ instead of $\mathbf{x}$).

If $x_{v_2}<x_{v_3}$ (resp. $x_{v_q}<x_{v_3}$), then we can construct a new signed graph $\Gamma'$ from $\Gamma$ by reversing the sign of the positive edge $v_4v_2$ (resp. $v_4v_q$) and the negative edge $v_4v_3$ (resp. $v_4v_3$) whose negative edges also form a bicyclic graph $B'$, then by Lemma \ref{lem:perturbation}, $\lambda_1(\Gamma)<\lambda_1(\Gamma')$, which is a contradiction.
Thus $x_{v_3}\leqslant x_{v_2}$ and $x_{v_3}\leqslant x_{v_q}$.
If $x_{v_1}>0$ (resp. $x_{v_p}>0$), then we can construct a new signed graph $\Gamma'$ from $\Gamma$ by reversing the sign of the positive edge $v_1v_3$ (resp. $v_pv_3$) and the negative edge $v_1v_2$ (resp. $v_pv_q$) whose negative edges also form a bicyclic graph $B'$, then by Lemma \ref{lem:perturbation}, $\lambda_1(\Gamma)<\lambda_1(\Gamma')$, which is a contradiction. Thus we have $x_{v_1}\leqslant 0$ and $x_{v_p}\leqslant 0$.

Next we will prove that $x_{v_1}=x_{v_2}=x_{v_3}=x_{v_4}=0$.
Since $x_{v_4}\geqslant0\geqslant x_{v_1}$ and $x_{v_p}\leqslant 0$, we can construct a new signed graph $\Gamma'$ obtained from $\Gamma$ by reversing the sign of the positive edge $v_pv_4$ and the negative edge $v_pv_1$ whose negative edges also form a bicyclic graph $B'$, by Lemma \ref{lem:perturbation}, we have $\lambda_1(\Gamma)\leqslant\lambda_1(\Gamma')$. If one of the inequalities $x_{v_4}\geqslant x_{v_1}$ and $x_{v_p}\leqslant 0$ is strict, then we have $\lambda_1(\Gamma)<\lambda_1(\Gamma')$, which is a contradiction. Thus
$x_{v_4}=x_{v_1}=x_{v_p}=0$. If $x_{v_3}\neq0$ (resp. $x_{v_2}\neq0$), then we can construct a new signed graph $\Gamma'$ from $\Gamma$ by reversing the sign of the positive edge $v_3v_1$ (resp. $v_2v_4$) and the negative edge $v_3v_4$ (resp. $v_2v_1$) whose negative edges also form a bicyclic graph $B'$, then by Lemma \ref{lem:perturbation}, $\lambda_1(\Gamma)<\lambda_1(\Gamma')$, which is also a contradiction. Thus $x_{v_2}=x_{v_3}=0$, and then $x_{v_1}=x_{v_2}=x_{v_3}=x_{v_4}=0$, which is a contradiction by Lemma \ref{Lem:Cl1Cl2}.

Let $N^r_B(\hat{B})$ be the set of vertices at distance $r$ from the vertices of $\hat{B}$ in $B$. Since $\hat{B}\cong\hat{\theta}(2,2,2)$, without loss of generality, suppose $C_{l_1}=v_1v_2v_3u_1v_1$, then  $C_{l_2}=v_1v_2v_3\omega_1v_1$ or $C_{l_2}=v_1u_1v_3\omega_1v_1$ (see Figure \ref{fig:hatB1hatB2}).

\noindent\textbf{Claim 2.} $N^2_B(\hat{B})=\emptyset$ and $N^1_B(\hat{B})$ is an independent set.

By the contrary, suppose that $v_q\in N^2_B(\hat{B})$. Without loss of generality, we consider the cycle $C_{l_1}$. Suppose that $v_1v_p$, $v_pv_q\in E(B)\backslash E(\hat{B})$, where $v_p\in N^1_B(\hat{B})$. Then we know that $v_p$ is not adjacent to $v_3$, $v_p$ is not adjacent to $u_1$, and $v_q$ is not adjacent to $u_1$. By a similar argument to Claim 1, we can get a same contradiction. Then $v_p$ is pendant vertex in $B$. Hence, $N^2_B(C_{l_1})=\emptyset$. Similarly, $N^2_B(C_{l_2})=\emptyset$. Then we have $N^2_B(\hat{B})=\emptyset$.

Since $\hat{B}\cong\hat{\theta}(2,2,2)$, $\hat{B}$ consists of two vertex-induced cycles of length $4$. Then we know that $N^1_B(\hat{B})$ is an independent set.

\noindent\textbf{Claim 3.} For any two non-adjacent vertices of $C_{l_i}$ where $i=1,2$, there is at most one vertex of them that has neighbors in $V(B)\backslash V(\hat{B})$.

Without loss of generality, we consider the cycle $C_{l_1}$. Suppose to the contrary that $v_1v_p,v_3v_q\in E(B)\backslash E(\hat{B})$, where $v_p,v_q\in V(B)\backslash V(\hat{B})$. Suppose $x_{u_1}\geqslant 0$ (otherwise, consider $-\mathbf{x}$ instead of $\mathbf{x}$).

If $x_{v_p}<x_{v_3}$ (resp. $x_{v_q}<x_{v_1}$), then we can construct a new signed graph $\Gamma'$ from $\Gamma$ by reversing the sign of the positive edge $u_1v_p$ (resp. $u_1v_q$) and the negative edge $u_1v_3$ (resp. $u_1v_1$) whose negative edges also form a bicyclic graph $B'$, then by Lemma \ref{lem:perturbation}, $\lambda_1(\Gamma)<\lambda_1(\Gamma')$, which is a contradiction. Thus $ x_{v_3}\leqslant x_{v_p}$ and $x_{v_1}\leqslant x_{v_q} $. If $x_{v_1}>0$ (resp. $x_{v_3}>0$), then we can construct a new signed graph $\Gamma'$ from $\Gamma$ by reversing the sign of the positive edge $v_1v_3$ (resp. $v_3v_1$) and the negative edge $v_1v_p$ (resp. $v_3v_q$) whose negative edges also form a bicyclic graph $B'$, then by Lemma \ref{lem:perturbation}, $\lambda_1(\Gamma)<\lambda_1(\Gamma')$, which is a contradiction. Thus $x_{v_1}\leqslant 0$ and $x_{v_3}\leqslant 0$. If $x_{v_p}<x_{v_q}$ (resp. $x_{v_p}>x_{v_q}$), then a new signed graph $\Gamma'$ can be obtained from $\Gamma$ by reversing the sign of the positive edge $v_1v_q$ (resp. $v_3v_p$) and the negative edge $v_1v_p$ (resp. $v_3v_q$) whose negative edges also form a bicyclic graph $B'$, by Lemma \ref{lem:perturbation}, we have $\lambda_1(\Gamma)<\lambda_1(\Gamma')$, which is a contradiction. Thus $x_{v_p}=x_{v_q}$.

Next we will prove that $x_{v_1}=x_{v_2}=x_{v_3}=x_{u_1}=0$. Since $x_{v_1}\leqslant 0$ and $x_{v_p}=x_{v_q}$, we can construct a new signed graph $\Gamma'$ from $\Gamma$ by reversing the sign of the positive edge $v_1v_q$ and the negative edge $v_1v_p$ whose negative edges also form a bicyclic graph $B'$, then by Lemma \ref{lem:perturbation}, $\lambda_1(\Gamma)\leqslant\lambda_1(\Gamma')$.  If $x_{v_1}<0$, then $\lambda_1(\Gamma)<\lambda_1(\Gamma')$, which is a contradiction. Thus we have $x_{v_1}=0$. Similarly, we have $x_{v_3}=0$. Since $x_{v_1}=x_{v_3}=0$, if $x_{v_p}\neq0$, then we can construct a new signed graph $\Gamma'$ from $\Gamma$ by reversing the sign of the positive edge $v_pv_3$ and the negative edge $v_pv_1$ whose negative edges also form a bicyclic graph $B'$, then by Lemma \ref{lem:perturbation}, $\lambda_1(\Gamma)<\lambda_1(\Gamma')$, which is a contradiction. Thus $x_{v_p}=0$. If $x_{u_1}\neq0$, then we can construct a new signed graph $\Gamma'$ from $\Gamma$ by reversing the sign of the positive edge $u_1v_p$ and the negative edge $u_1v_1$ whose negative edges also form a bicyclic graph $B'$, then by Lemma \ref{lem:perturbation}, $\lambda_1(\Gamma)<\lambda_1(\Gamma')$, which is a contradiction. Thus $x_{u_1}=0$. Similarly, we have $x_{v_2}=0$. Then $x_{v_1}=x_{v_2}=x_{v_3}=x_{u_1}=0$, by Lemma \ref{Lem:Cl1Cl2}, which is a contradiction.

By a similar argument, we can deduce that there is at most one vertex in $\{v_2,u_1\}$ that has neighbors in $V(B)\backslash V(\hat{B})$. Moreover, considering the cycle $C_{l_2}$, we can also deduce that there is at most one vertex in $\{v_1,v_3\}$ and at most one vertex in $v_2,\omega_1$ or $\{u_1,\omega_1\}$ that has neighbors in $V(B)\backslash V(\hat{B})$.

\noindent\textbf{Claim 4.} $B\cong \theta_{1}(s,t)$ for some $s,t\geq0$ and $s+t=k-6$.

Since $\hat{B}$ consists of $C_{l_1}$ and $C_{l_2}$, by Claim 3, at most one vertex of $\{u_1,v_2,\omega_1\}$ has neighbors in $V(B)\backslash V(\hat{B})$ and at most one vertex of $\{v_1,v_3\}$ has neighbors in $V(B)\backslash V(\hat{B})$. By Claim 2, we have $B\cong \theta_{1}(s,t)$ (see Figure \ref{fig:B11}) for some $s,t\geqslant0$ and $s+t=k-6$.

Combining Lemma \ref{lem:BBBst1} and Lemma \ref{lem:BBBmax}, $\lambda_1((K_n,B))$ does not attain the maximum while $B\cong \theta_{1}(s,t)$. Then we have a contradiction.

\noindent\underline{\textbf{Case 2.}} $l_1=3$

\noindent\textbf{Claim 5.} $\hat{B}\cong\hat{\theta}(1,2,2)$

Suppose $C_{l_1}=v_1v_2v_3v_1$. Assume that $v_1v_p\in E(\hat{B})$, where $v_p\in V(\hat{B})\backslash V(C_{l_1})$.
If $v_p$ is adjacent to $v_2$ or $v_3$ in $\hat{B}$, it is obvious that $\hat{B}\cong\hat{\theta}(1,2,2)$. Suppose that $v_p$ is not adjacent to $v_2$ and $v_3$ in $\hat{B}$. There is another vertex $v_q$ in $V(\hat{B})$ that is adjacent to $v_p$ in $\hat{B}$. Suppose that $v_q$ is not adjacent to $v_2$ (otherwise, exchange the subscripts $v_3$ and $v_2$). Without loss of generality, suppose $x_{v_2}\geqslant 0$ (otherwise, consider $-\mathbf{x}$ instead of $\mathbf{x}$).

If $x_{v_p}<x_{v_1}$, then we can construct a new signed graph $\Gamma'$ from $\Gamma$ by reversing the sign of the positive edge $v_2v_p$ and the negative edge $v_2v_1$ whose negative edges also form a bicyclic graph $B'$, then by Lemma \ref{lem:perturbation}, $\lambda_1(\Gamma)<\lambda_1(\Gamma')$, which is a contradiction. Thus $x_{v_1}\leqslant x_{v_p}$.
If $x_{v_3}<0$, then we can construct a new signed graph $\Gamma'$ from $\Gamma$ by reversing the sign of the positive edge $v_3v_p$ and the negative edge $v_3v_1$ whose negative edges also form a bicyclic graph $B'$, then by Lemma \ref{lem:perturbation}, $\lambda_1(\Gamma)<\lambda_1(\Gamma')$, which is a contradiction. Thus $x_{v_3}\geqslant 0$.
If $x_{v_p}<x_{v_2}$, then we can construct a new signed graph $\Gamma'$ from $\Gamma$ by reversing the sign of the positive edge $v_3v_p$ and the negative edge $v_3v_2$ whose negative edges also form a bicyclic graph $B'$, then by Lemma \ref{lem:perturbation}, $\lambda_1(\Gamma)<\lambda_1(\Gamma')$, which is a contradiction. Thus $x_{v_p}\geqslant x_{v_2}\geqslant 0$.
If $x_{v_q}>0$, then we can construct a new signed graph $\Gamma'$ from $\Gamma$ by reversing the sign of the positive edge $v_qv_2$ and the negative edge $v_qv_p$ whose negative edges also form a bicyclic graph $B'$, then by Lemma \ref{lem:perturbation}, $\lambda_1(\Gamma)<\lambda_1(\Gamma')$, which is a contradiction. Thus $x_{v_q}\leqslant 0$.
Then we have $x_{v_3}\geqslant 0\geqslant x_{v_q}$.

Next we will prove that $x_{v_1}=x_{v_2}=x_{v_3}=0$.
Since $x_{v_3}\geqslant 0\geqslant x_{v_q}$ and $x_{v_2}\geqslant 0$, we can construct a new signed graph $\Gamma'$ obtained from $\Gamma$ by reversing the sign of the positive edge $v_2v_q$ and the negative edge $v_2v_3$ whose negative edges also form a bicyclic graph $B'$, by Lemma \ref{lem:perturbation}, we have $\lambda_1(\Gamma)\leqslant\lambda_1(\Gamma')$. If one of the inequalities $x_{v_3}\geqslant x_{v_q}$ and $x_{v_2}\geqslant 0$ is strict, then we have $\lambda_1(\Gamma)<\lambda_1(\Gamma')$, which is a contradiction. Thus $x_{v_3}=x_{v_q}=x_{v_2}=0$. If $x_{v_1}\neq 0$, then we can construct a new signed graph $\Gamma'$ from $\Gamma$ by reversing the sign of the positive edge $v_2v_q$ and the negative edge $v_2v_1$ whose negative edges also form a bicyclic graph $B'$, then by Lemma \ref{lem:perturbation}, $\lambda_1(\Gamma)<\lambda_1(\Gamma')$, which is a contradiction. Hence we have $x_{v_1}=0$ and then $x_{v_1}=x_{v_2}=x_{v_3}=0$, which is a contradiction by Lemma \ref{Lem:Cl1Cl2}.

Let $N^r_B(\hat{B})$ be the set of vertices at distance $r$ from the vertices of $\hat{B}$ in $B$. Since $\hat{B}\cong\hat{\theta}(1,2,2)$, without loss of generality, suppose $V(C_{l_1})=\{v_1,v_2,u_1\}$ and  $V(C_{l_2})=\{v_1,v_2,\omega_1\}$ (see Figure \ref{fig:hatB1hatB2}).

\noindent\textbf{Claim 6.} $N^2_B(\hat{B})=\emptyset$ and $N^1_B(\hat{B})$ is an independent set.

By the contrary, suppose that $v_q\in N^2_B(\hat{B})$. Without loss of generality, we consider the cycle $C_{l_1}$. Suppose that $v_1v_p,v_pv_q\in E(B)$, where $v_p\in N^1_B(\hat{B})$. Obviously, $v_1v_p,v_pv_q\notin E(\hat{B})$. Then we have $v_p$ is not adjacent to $v_2$, and $v_q$ is not adjacent to $v_2$. By a similar argument to Claim 5, we can get a same contradiction. Then $v_p$ is pendant vertex in $B$. Hence, $N^2_B(\hat{B})=\emptyset$.

Since $\hat{B}\cong\hat{\theta}(1,2,2)$, $\hat{B}$ consists of two vertex-induced cycles of length $3$. Then we know that $N^1_B(\hat{B})$ is an independent set.

\noindent\textbf{Claim 7.} At most one vertex of $C_{l_i}$ has neighbors in $V(B)\backslash V(\hat{B})$, where $i=1,2$.

Without loss of generality, we consider the cycle $C_{l_1}$. Suppose to the contrary that $v_1v_p,v_2v_q\in E(B)$, where $v_p,v_q\in V(B)\backslash V(\hat{B})$. Suppose $x_{u_1}\geqslant0$ (otherwise, consider $-\mathbf{x}$ instead of $\mathbf{x}$).

If $x_{v_q}<x_{v_1}$ (resp. $x_{v_p}<x_{v_2}$), then we can construct a new signed graph $\Gamma'$ from $\Gamma$ by reversing the sign of the positive edge $u_1v_q$ (resp. $u_1v_p$) and the negative edge $u_1v_1$ (resp. $u_1v_2$) whose negative edges also form a bicyclic graph $B'$, then by Lemma \ref{lem:perturbation}, $\lambda_1(\Gamma)<\lambda_1(\Gamma')$, which is a contradiction. Thus $x_{v_1}\leqslant x_{v_q}$ and $x_{v_2}\leqslant x_{v_p}$. And then, if $x_{v_p}<0$ (resp. $x_{v_q}<0$), then we can construct a new signed graph $\Gamma'$ from $\Gamma$ by reversing the sign of the positive edge $v_pv_q$ (resp. $v_qv_p$) and the negative edge $v_pv_1$ (resp. $v_qv_2$) whose negative edges also form a bicyclic graph $B'$, then by Lemma \ref{lem:perturbation}, $\lambda_1(\Gamma)<\lambda_1(\Gamma')$, which is a contradiction. Thus $x_{v_p}\geqslant0$ and $x_{v_q}\geqslant0$.

If $x_{v_1}> x_{v_2}$, then a new signed graph $\Gamma'$ can be obtained from $\Gamma$ by reversing the sign of the positive edge $v_pv_2$ and the negative edge $v_pv_1$ whose negative edges also form a bicyclic graph $B'$, by Lemma \ref{lem:perturbation}, we have $\lambda_1(\Gamma)<\lambda_1(\Gamma')$, which is a contradiction. Thus $x_{v_1}\leqslant x_{v_2}$. However, if the inequality $x_{v_1}\leqslant x_{v_2}$ is strict, then we can construct a new signed graph $\Gamma'$ obtained from $\Gamma$ by reversing the sign of the positive edge $v_qv_1$ and the negative edge $v_qv_2$ whose negative edges also form a bicyclic graph $B'$, by Lemma \ref{lem:perturbation}, we have $\lambda_1(\Gamma)<(\Gamma')$, which is a contradiction. Thus $x_{v_1}=x_{v_2}$.
If $x_{v_p}>0$ (resp. $x_{v_q}>0$), then we can construct a new signed graph $\Gamma'$ from $\Gamma$ by reversing the sign of the positive edge $v_pv_2$ (resp. $v_qv_1$) and the negative edge $v_pv_1$ (resp. $v_qv_2$) whose negative edges also form a bicyclic graph $B'$, then by Lemma \ref{lem:perturbation}, $\lambda_1(\Gamma)<\lambda_1(\Gamma')$, which is a contradiction. Thus $x_{v_p}=x_{v_q}=0$.

Next we will prove that $x_{v_1}=x_{v_2}=x_{u_1}=0$.
If $x_{v_1}=x_{v_2}\neq0$, since $x_{v_p}=0$, then we can construct a new signed graph $\Gamma'$ from $\Gamma$ by reversing the sign of the positive edge $v_2v_p$ and the negative edge $v_2v_1$ whose negative edges also form a bicyclic graph $B'$, then by Lemma \ref{lem:perturbation}, $\lambda_1(\Gamma)<\lambda_1(\Gamma')$, which is a contradiction. Thus $x_{v_1}=x_{v_2}=0$. If $x_{u_1}\neq0$, then we can construct a new signed graph $\Gamma'$ from $\Gamma$ by reversing the sign of the positive edge $u_1v_q$ and the negative edge $u_1v_2$ whose negative edges also form a bicyclic graph $B'$, then by Lemma \ref{lem:perturbation}, $\lambda_1(\Gamma)<\lambda_1(\Gamma')$, which is a contradiction. Thus $x_{u_1}=0$. Then we have $x_{v_1}=x_{v_2}=x_{u_1}=0$, by Lemma \ref{Lem:Cl1Cl2}, which is a contradiction. Then there is at most one vertex in $V(C_{l_1})$ has neighbors in $V(B)\backslash V(\hat{B})$. Similarly, there is at most one vertex in $V(C_{l_2})$ has neighbors in $V(B)\backslash V(\hat{B})$.

\noindent\textbf{Claim 8.} $B\cong \theta_2(0,k-5)$

Since $\hat{B}$ consists of $C_{l_1}$ and $C_{l_2}$, by Claim 7, at most one vertex of $\{v_1,v_2,u_1\}$ has neighbors in $V(B)\backslash V(\hat{B})$ and at most one vertex of $\{v_1,v_2,\omega_1\}$ has neighbors in $V(B)\backslash V(\hat{B})$. If  $u_1$ and $\omega_1$ have no neighbors in $V(B)\backslash V(\hat{B})$, it is obvious that $B\cong \theta_2(0,k-5)$, by Claim 6.
Next, we will prove that there is a contradiction if $u_1$ or $\omega_1$ has neighbors in $V(B)\backslash V(\hat{B})$. Without loss of generality, we consider the vertex $u_1$. Suppose that $u_1v_t\in E(B)$, where $v_t\in V(B)\backslash V(\hat{B})$ and $x_{v_t}\geqslant0$ (otherwise we consider $-\mathbf{x}$ instead of $\mathbf{x}$).

If $x_{u_1}>x_{v_2}$, then we can construct a new signed graph $\Gamma'$ from $\Gamma$ by reversing the sign of the positive edge $v_tv_2$ and the negative edge $v_tu_1$ whose negative edges also form a bicyclic graph $B'$, then by Lemma \ref{lem:perturbation}, $\lambda_1(\Gamma)<\lambda_1(\Gamma')$, which is a contradiction. Thus $x_{u_1}\leqslant x_{v_2}$.
If $x_{\omega_1}>0$, then we can construct a new signed graph $\Gamma'$ from $\Gamma$ by reversing the sign of the positive edge $\omega_1u_1$ and the negative edge $\omega_1v_2$ whose negative edges also form a bicyclic graph $B'$, then by Lemma \ref{lem:perturbation}, $\lambda_1(\Gamma)<\lambda_1(\Gamma')$, which is a contradiction. Thus $x_{\omega_1}\leqslant0$. Since $x_{v_t}\geqslant0\geqslant x_{\omega_1}$, if $x_{v_1}<0$, then we can construct a new signed graph $\Gamma'$ from $\Gamma$ by reversing the sign of the positive edge $v_1v_t$ and the negative edge $v_1\omega_1$ whose negative edges also form a bicyclic graph $B'$, then by Lemma \ref{lem:perturbation}, $\lambda_1(\Gamma)<\lambda_1(\Gamma')$, which is a contradiction. Thus $x_{v_1}\geqslant 0$.

If $x_{v_t}>x_{v_2}$, then a new signed graph $\Gamma'$ can be obtained from $\Gamma$ by reversing the sign of the positive edge $\omega_1v_t$ and the negative edge $\omega_1v_2$ whose negative edges also form a bicyclic graph $B'$, by Lemma \ref{lem:perturbation}, we have $\lambda_1(\Gamma)<\lambda_1(\Gamma')$, which is a contradiction. Thus $x_{v_t}\leqslant x_{v_2}$. However, if the inequality $x_{v_t}\leqslant x_{v_2}$ is strict, then we can construct a new signed graph $\Gamma'$ can be obtained from $\Gamma$ by reversing the sign of the positive edge $v_1v_t$ and the negative edge $v_1v_2$ whose negative edges also form a bicyclic graph $B'$, by Lemma \ref{lem:perturbation}, we have $\lambda_1(\Gamma)<\lambda_1(\Gamma')$, which is a contradiction. Thus, $x_{v_t}=x_{v_2}$.

Next we will prove that $x_{\omega_1}=x_{v_1}=x_{v_2}=0$.
If $x_{\omega_1}\neq0$ (resp. $x_{v_1}\neq0$), then we can construct a new signed graph $\Gamma'$ from $\Gamma$ by reversing the sign of the positive edge $\omega_1v_t$ (resp. $v_1v_t$) and the negative edge $\omega_1v_2$ (resp. $v_1v_2$) whose negative edges also form a bicyclic graph $B'$, then by Lemma \ref{lem:perturbation}, $\lambda_1(\Gamma)<\lambda_1(\Gamma')$, which is a contradiction. Thus $x_{\omega_1}=x_{v_1}=0$. If $x_{u_1}\neq0$, then we can construct a new signed graph $\Gamma'$ from $\Gamma$ by reversing the sign of the positive edge $u_1\omega_1$ and the negative edge $u_1v_1$ whose negative edges also form a bicyclic graph $B'$, then by Lemma \ref{lem:perturbation}, $\lambda_1(\Gamma)<\lambda_1(\Gamma')$, which is a contradiction. Thus $x_{u_1}=0$. If $x_{v_t}\neq0$, then we can construct a new signed graph $\Gamma'$ from $\Gamma$ by reversing the sign of the positive edge $v_tv_1$ and the negative edge $v_tu_1$ whose negative edges also form a bicyclic graph $B'$, then by Lemma \ref{lem:perturbation}, $\lambda_1(\Gamma)<\lambda_1(\Gamma')$, which is a contradiction. Thus $x_{v_2}=x_{v_t}=0$. Then we have $x_{\omega_1}=x_{v_1}=x_{v_2}=0$, by Lemma \ref{Lem:Cl1Cl2}, which is a contradiction. Therefore, $u_1$ has no neighbor in $V(B)\backslash V(\hat{B})$. Similarly, $\omega_1$ has the same situation.

By combining Lemma \ref{lem:BBBst1} and Lemma \ref{lem:BBBmax}, we can establish that among the signed complete graphs with $n$ vertices and $k$ negative edges, $(K_n,B)$ has the maximum index if and only if $B$ is isomorphic to $\theta_{2}(0,k-5)$, when $n\geqslant k+20$ and $k\leqslant15$.
\end{proof of 1.1}

\section*{Declaration}
The authors have declared that no competing interest exists.

\appendix
\section{}\label{App:A}

Our goal in this section is to complete the proofs of Lemma \ref{lem:BBBst1} and Lemma \ref{lem:BBBmax}.
Let
\begin{small}
\begin{equation*}
\begin{split}
p(\lambda)=&(s-t-2)\lambda^4+4(s-t-2)\lambda^3-2[(u+1)(s-t-2)+4]\lambda^2\\
&-4[(u+3)(s-t-2)+4]\lambda+14su-7s-14tu+7t-12u+6\\
=&p_1(\lambda)+p_2(\lambda)=P_{1}(\lambda)+P_{2}(\lambda),
\end{split}
\end{equation*}
\end{small}
where
\begin{small}
$$p_1(\lambda)=(s-t-2)\lambda^4+4(s-t-2)\lambda^3-8\lambda^2-16\lambda-7s-14tu-12u,$$ $$p_2(\lambda)=-2(u+1)(s-t-2)\lambda^2-4(u+3)(s-t-2)\lambda+14su+7t+6,$$
$$P_{1}(\lambda)=-2[(u+1)(s-t-2)+4]\lambda^2-4[(u+3)(s-t-2)+4]\lambda-7s-14tu-12u,$$
\end{small}
and
\begin{small}
$$P_{2}(\lambda)=(s-t-2)\lambda^4+4(s-t-2)\lambda^3+14su+7t+6.$$
\end{small}

\begin{proof of 2.1}
Let $\Gamma=(K_n,\theta_{1}(s,t))$, $\Gamma'=(K_n,\theta_{1}(s-1,t+1))$. First we assume that $s\leqslant t+2$. By Lemma \ref{lem:BBB1}, we deduce
\begin{small}
\begin{equation*}
\begin{split}
\varphi(\Gamma',\lambda)-\varphi(\Gamma,\lambda)=8(\lambda+1)^{n-7}p(\lambda)=8(\lambda+1)^{n-7}(p_1(\lambda)+p_2(\lambda)).
\end{split}
\end{equation*}
\end{small}
Let $\lambda_1$ be the index of $\Gamma$, then we have
\begin{small}
\begin{equation*}
\begin{split}
\varphi(\Gamma',\lambda_1)=&8(\lambda_1+1)^{n-7}(p_1(\lambda_1)+p_2(\lambda_1)).
\end{split}
\end{equation*}
\end{small}
By the assumption $s\leqslant t+2$, we will know that $p_1(\lambda)$ is a decreasing function and $p_2(\lambda)$ is an increasing function. Since $\Gamma$ has $(K_{n-3},+)$ as an induced subgraph, by Theorem \ref{thm:interlacing}, we have $\lambda_1\geqslant n-4$. Then we have
$p_1(\lambda_1)\leqslant p_1(n-4)$.
Since $\lambda_1\leq n-1$, we have $p_2(\lambda_1)\leqslant p_2(n-1)$.
Since $k=s+t+6$, $n\geqslant k-1$ and $k\geqslant 15$, we have
\begin{small}
\begin{equation*}
\begin{split}
p_1(\lambda_1)+p_2(\lambda_1)\leqslant& p_1(n-4)+p_2(n-1)\\
=&(s-t-2)(k-15)n^3+(s-t-2)(2k^3+40k^2-158k+119)\\
&-8k^2+64k-128<0.
\end{split}
\end{equation*}
\end{small}
Hence $p_1(\lambda_1)+p_2(\lambda_1)<0$, $\varphi(\Gamma',\lambda_1)<0$ and  then $\lambda_1((K_n,\theta_{1}(s-1,t+1)))>\lambda_1((K_n,\theta_{1}(s,t)))$. Therefore, we have
\begin{small}
\begin{equation*}
\begin{split}
\lambda_1((K_n,\theta_{1}(0,k-6)))>\cdots>\lambda_1((K_n,\theta_{1}(\lfloor\frac{k-6}{2}\rfloor,\lceil\frac{k-6}{2}\rceil)))>\cdots>\lambda_1((K_n,\theta_1(k-6,0))).
\end{split}
\end{equation*}
\end{small}
Thus, $\lambda_1((K_n,\theta_{1}(0,k-6)))\geqslant\lambda_1((K_n,\theta_{1}(s,t)))$ for each pair $s,t\geqslant0$ and $s+t=k-6$.

Next we consider $s>t+2$. Let $\lambda_*$ be the index of $\Gamma'$. Since
\begin{small}
\begin{equation*}
\begin{split}
\varphi(\Gamma,\lambda)-\varphi(\Gamma',\lambda)=-8(\lambda+1)^{n-7}p(\lambda)=-8(\lambda+1)^{n-7}(P_{1}(\lambda)+P_{2}(\lambda)),
\end{split}
\end{equation*}
\end{small}
then we have
\begin{small}
\begin{equation*}
\begin{split}
\varphi(\Gamma,\lambda_*)=-8(\lambda_*+1)^{n-7}(P_{1}(\lambda_*)+P_{2}(\lambda_*)).
\end{split}
\end{equation*}
\end{small}
Since $\Gamma'$ has $(K_{n-3},+)$ as an induced subgraph, by Theorem \ref{thm:interlacing}, we have $\lambda_*\geqslant n-4$. Then we have $P_{2}(\lambda_*)\geqslant P_{2}(n-4)$. Since $\lambda_*\leqslant n-1$, we have $P_{1}(\lambda_*)\geqslant P_{1}(n-1)$.
Since $k=s+t+6$ and $n\geqslant k-1$, we have
\begin{small}
\begin{equation*}
\begin{split}
P_{1}(\lambda_*)+P_{2}(\lambda_*)\geqslant& P_{1}(n-1)+P_{2}(n-4)\\
=&(k-15)(s-t-2)n^3+(2k^2+34k-108)n\\
&+16(s-t-1)u+3s-3t-6>0.
\end{split}
\end{equation*}
\end{small}
Hence $P_{1}(\lambda_*)+P_{2}(\lambda_*)>0$, $\varphi(\Gamma',\lambda_1)<0$ and then $\lambda_*((K_n,\theta_{1}(s,t)))>\lambda_*((K_n,\theta_{1}(s-1,t+1)))$. Therefore, we have
\begin{small}
\begin{equation*}
\begin{split}
\lambda_*((K_n,\theta_{1}(k-6,0)))>\cdots>\lambda_*((K_n,\theta_{1}(\lceil\frac{k-6}{2}\rceil,\lfloor\frac{k-6}{2}\rfloor)))>\cdots>\lambda_*((K_n,\theta_1(0,k-6))).
\end{split}
\end{equation*}
\end{small}
Thus, $\lambda_*((K_n,\theta_{1}(k-6,0)))\geqslant\lambda_*((K_n,\theta_{1}(s,t)))$ for each pair $s,t\geqslant0$ and $s+t=k-6$. Hence the proof is complete.
\end{proof of 2.1}
To simplify representation, let
\begin{small}
\begin{equation*}
\begin{split}
s(\lambda)=&(k-8)\lambda^4+4(n-7)\lambda^3-2(-k^2+kn+12k-14n+10)\lambda^2\\
&-4(-k^2+kn+11k-11n-3)\lambda-2kn-21k-20n+2k^2+12\\
=&s_1(\lambda)+s_2(\lambda),\\
S(\lambda)=&(k-7)\lambda^4+2(n+k-13)\lambda^3-(-k^2+kn+8k-10n+11)\lambda^2\\
&-2(-2k^2+2kn+17k-17n-3)\lambda-2kn-17k-4n+2k^2+9\\
=&S_{1}(\lambda)+S_{2}(\lambda),
\end{split}
\end{equation*}
\end{small}
where
\begin{small}
\begin{equation*}
\begin{split}
s_1(\lambda)=&(k-8)\lambda^4+4(n-7)\lambda^3+2k^2+12,\\ s_2(\lambda)=&-2(-k^2+kn+12k-14n+10)\lambda^2-4(-k^2+kn+11k-11n-3)\lambda\\
&-2kn-21k-20n,\\
S_{1}(\lambda)=&(k-7)\lambda^4+2(n+k-13)\lambda^3+2k^2+9,\\ S_{2}(\lambda)=&-(-k^2+kn+8k-10n+11)\lambda^2-2(-2k^2+2kn+17k-17n-3)\lambda\\
&-2kn-17k-4n.
\end{split}
\end{equation*}
\end{small}

\begin{proof of 2.2}
(i) Suppose that $\Gamma=(K_n,\theta_{1}(0,k-6))$ and $\Gamma'=(K_n,\theta_{2}(0,k-5))$. By Lemma \ref{lem:BBB1} and Lemma \ref{lem:BBB2}, we deduce
\begin{small}
\begin{equation*}
\begin{split}
\varphi(\Gamma',\lambda)-\varphi(\Gamma,\lambda)=-8(\lambda+1)^{n-7}s(\lambda)=-8(\lambda+1)^{n-7}(s_1(\lambda)+s_2(\lambda)).
\end{split}
\end{equation*}
\end{small}
Let $\lambda_1$ be the index of $\Gamma$, then we have
\begin{small}
\begin{equation*}
\begin{split}
\varphi(\Gamma',\lambda_1)=-8(\lambda_1+1)^{n-7}(s_1(\lambda_1)+s_2(\lambda_1)).
\end{split}
\end{equation*}
\end{small}
Since $(K_n,\theta_{1}(0,k-6))$ has $(K_{n-2},+)$ as an induced subgraph, by Theorem \ref{thm:interlacing} we have $\lambda_1\geqslant n-3$, and $k\geqslant15$, $n\geqslant k+20$. Then$s_1(\lambda_1)\geqslant s_1(n-3)$. Since $\lambda_1\leqslant n-1$, we have $s_2(\lambda_1)\geqslant s_2(n-1)$.
Since $k=s+t+6$ and $n\geqslant k-1$, we have
\begin{small}
\begin{equation*}
\begin{split}
s_1(\lambda_1)+s_2(\lambda_1)&\geqslant s_1(n-3)+s_2(n-1)\\
&=(k^2-19k+64)n^3+(2k^3+28k^2-238k+120)n+80k+88>0.
\end{split}
\end{equation*}
\end{small}
Hence $s_1(\lambda_1)+s_2(\lambda_1)>0$, $\varphi(\Gamma',\lambda_1)<0$. And then $\Gamma'$ has a larger index, which means that $\lambda_1((K_n,\theta_{1}(0,k-6)))<\lambda_1((K_n,\theta_{2}(0,k-5)))$.

(ii) Suppose that $\Gamma=(K_n,\theta_{1}(k-6,0))$ and $\Gamma'=(K_n,\theta_{2}(0,k-5))$. By Lemma \ref{lem:BBB1} and Lemma \ref{lem:BBB2}, we deduce
\begin{small}
\begin{equation*}
\begin{split}
\varphi(\Gamma',\lambda)-\varphi(\Gamma,\lambda)=-16(\lambda+1)^{n-7}S(\lambda)=-16(\lambda+1)^{n-7}(S_{1}(\lambda)+S_{2}(\lambda)).
\end{split}
\end{equation*}
\end{small}
Let $\lambda_1$ be the index of $\Gamma$, then we have
\begin{small}
\begin{equation*}
\begin{split}
\varphi(\Gamma',\lambda_1)=-16(\lambda_1+1)^{n-7}(S_{1}(\lambda_1)+S_{2}(\lambda_1)).
\end{split}
\end{equation*}
\end{small}
Since $(K_n,\theta_{1}(k-6,0))$ has $(K_{n-3},+)$ as an induced subgraph, by Theorem \ref{thm:interlacing} we have $\lambda_1\geqslant n-4$, and $n\geqslant k+20$.
We have $S_{1}(\lambda_1)\geqslant S_{1}(n-4)$.
Since $\lambda_1\leqslant n-1$, we have $S_{2}(\lambda_1)\geqslant S_{2}(n-1)$.
Since $k=s+t+6$ and $n\geqslant k-1$, we have
\begin{small}
\begin{equation*}
\begin{split}
S_{1}(\lambda_1)+S_{2}(\lambda_1)&\geqslant S_{1}(n-4)+S_{2}(n-1)\\
&=(k-5)n^4+(82-16k)n^3+(2k^2+56k-292)n^2+(448-162k)n+129k-147\\
&>(k^2-k-18)n^3+2(k^3+27k^2-255k+370)n+129k-147>0.
\end{split}
\end{equation*}
\end{small}
Hence $S_{1}(\lambda_1)+S_{2}(\lambda_1)>0$, $\varphi(\Gamma',\lambda_1)<0$. And then $\Gamma'$ has a larger index, which means that $\lambda_1((K_n,\theta_{1}(k-6,0)))<\lambda_1((K_n,\theta_{2}(0,k-5)))$.
\end{proof of 2.2}

	
\end{document}